\documentclass[11pt]{article}
\usepackage[english]{babel}
\usepackage{amssymb}
\usepackage{amsthm}
\usepackage{amsmath}
\usepackage{esvect}
\usepackage{fullpage}
\usepackage{setspace}
\usepackage{hyperref}
\usepackage{verbatim}
\usepackage[dvips]{graphicx}
\usepackage{color}
\usepackage{geometry}
\usepackage{framed}

\newtheorem{thm}{Theorem}

\newtheorem{lemma}[thm]{Lemma}

\newtheorem{prop}[thm]{Proposition}

\newtheorem{defn}[thm]{Definition}

\newtheorem{clm}[thm]{Claim}
\newtheorem{obs}[thm]{Observation}

\newcommand\cQ{{\mathcal Q}}

\newcommand{\ignore}[1]{}

\title{REV - Majority problems of large query size}

\begin{document}

\author{
D\'aniel Gerbner\thanks{Research supported by the J\'anos Bolyai Research Fellowship of the Hungarian Academy of Sciences.} \thanks{Research supported by the National Research,
Development and Innovation Office NKFIH, grant K116769.}
\and
M\'at\'e Vizer\thanks{Research supported by the National Research, Development and Innovation
Office NKFIH, grant SNN 116095.}}

\date{MTA R\'enyi Institute \\
Hungary H-1053, Budapest, Re\'altanoda utca 13-15.\\
\small \texttt{gerbner@renyi.hu, vizermate@gmail.com}\\
\today}

\maketitle
\begin{abstract}

We study two models of the Majority problem. We are given $n$ balls and an unknown coloring of them with two colors. We can ask sets of balls of size $k$ as queries, and in the so-called General Model the answer to a query shows if all the balls in the set are of the same color or not. In the so-called Counting Model the answer to a query gives the difference between the cardinalities of the color classes in the query.

Our goal is to show a ball of the larger color class, or prove that the color classes are of the same size, using as few queries as possible. In this paper we improve the bounds given by De Marco and Kranakis \cite{DK2015} for the number of queries needed.

\end{abstract}

\section{Introduction}

\indent

We are given $n$ indexed balls - we identify them with the set $[n]=\{1,2,...,n\}$ - as an input, each colored in some way unknown to us, and we can ask subsets of $[n]$, that we call \textit{queries}. A ball $i \in [n]$ is called \textit{majority ball} if there are more than $n/2$ balls in the input set that have the same color as $i$. We call the problem of finding a majority ball (or showing that there is no such ball) the \textit{Majority Problem}. We would like to determine the minimum number of queries needed in the worst case, when our adversary - we will call him Adversary in the following -, who tells us the answers for the queries wants to postpone the solution of the Majority Problem. 


A \textit{model} of the Majority Problem is given by the number of balls $[n]$, the number of colors, the size of the queries (that we will denote by $k$), and the possible answers of Adversary. Sometimes we will use the hypergraph language, so we will speak about the hypergraph of the queries (i.e. the hypergraph, where the vertices are the balls and the edges are the asked queries up to a certain round.).

\vspace{2mm}

The first Majority Problem model, the so-called \textit{pairing model} - when the query size is two, and the answer of Adversary is yes if the two balls have the same color and no otherwise - was investigated by Fisher and Salzberg \cite{FS1982}, who proved that if we do not have any restriction on the number of colors, $\lceil 3n/2 \rceil - 2$ queries are necessary and sufficient to solve the Majority Problem. If the number of colors is two, then Saks and Werman \cite{SW1991} proved that the minimum number of queries needed is $n-b(n)$, where $b(n)$ is the number of $1$'s in the diadic form of $n$ (we note that there are simpler proofs of this result, see \cite{A2004,W2002}). In this paper we deal with two possible generalizations of the pairing model, the \textit{Counting} and the \textit{General} model. Both of them deals with queries of size greater than two.

The first model that generalized the pairing model to larger queries was introduced and investigated by De Marco, Kranakis and Wiener \cite{DKW2011}, then many results appeared in the literature  \cite{B2014,DK2015,EH2015,GKPPVW2015}. We note that for $k \ge 3$ it is possible that in some model one can not solve the Majority Problem for small $n$ or even for any $n$ (see such a model in \cite{GKPPVW2015}).
We also note that there are other possibilities to generalize the pairing model of the Majority Problem, e.g. the Plurality Problem \cite{A2004, ADM2005, CGMY2005, GKPP2013, GLV2017}, some random scenarios \cite{DP2006,DJKKS2007,GSU2016}, 
or investigate the case when Adversary can lie \cite{A2004,BGM2010}, etc.

\vspace{2mm}

\textbf{Structure of the paper} The rest of the paper is organized as follows: in the next section first we define the models and state the known results, then we state our new results. In Section 3 and Section 4 we prove our theorems, and in Section 5 we finish the article with some remarks.

\section{Models and results}
De Marco and Kranakis \cite{DK2015} introduced three different models, and gave upper and lower bounds for all of them. Eppstein and Hirschberg \cite{EH2015} improved three of those  six bounds, here we improve the other three.

In the following models we are given $n$ indexed balls, colored with two colors, and we ask queries of size $k$ ($k \ge 2$) that we denote by $Q$. We will give each model an abbreviation that we indicate after its name. If the abbreviation is $X$, then we will denote by $A(X,k,n)$ the number of queries needed to ask in the worst case of that model. At each model we mention the papers that deal with that model and write the best result that was known.

\subsection{Models, known results}

\vspace{3mm}

\textsc{Counting Model = CM}, \cite{DK2015,EH2015}:

\vspace{5mm}

$\bullet$ $\mathbf{Answer:}$ is a number $i \le k/2$ such that the query has exactly $i$ balls of one color and $k-i$ of the other.
(No indication is provided to us about the colors of each type.)

\vspace{3mm}

\begin{thm}\label{eh} (Eppstein, Hirschberg \cite{EH2015}, Theorem 2, Theorem 3, Theorem 4) For all $2 \le k \le n$ we have

$$f(n,k) \le A(CM,k,n) \le \frac{n}{\lceil \frac{k}{2} \rceil}  + O(k),$$

where

\begin{displaymath}
f(n,k):=
\left\{ \begin{array}{l l}
\lceil \frac{n-1}{k-1} \rceil & \textrm{if $k$ is odd,}\\
\frac{n}{k-1} - O(n^{1/3}) & \textrm{if $k$ is even.}\\
\end{array}
\right.
\end{displaymath}

\end{thm}

\vspace{1cm}

\noindent
\textsc{General (or Yes-No) Model = GM}, \cite{DK2015}:

\vspace{5mm}

$\bullet$ $\mathbf{Answer:}$ yes if all balls have the same color in $Q$, no otherwise.

\vspace{3mm}

\begin{thm}\label{dmk} (De Marco, Kranakis \cite{DK2015}, Theorem 5.1, Theorem 5.4) For all $2 \le k,n$ with $2k-1 \le n$, we have 

$$\Big{\lceil} \frac{n}{k} \Big{\rceil} \le A(GM,k,n) \le n-k + \binom{2k-1}{k}.$$

\end{thm}

\begin{thm}\label{borz} (Borzyszkowski \cite{DK2015}, Theorem 1) For all $2 \le k,n$ with $2k-1 \le n$ we have 

$$\lfloor n/2\rfloor+k-2 \le A(GM,k,n).$$

\end{thm}

We just mention here that there is a third model in \cite{DK2015}, where the answer gives the partition of the balls into the two color classes (but does not tell which part belongs to each color). The lower bound by Eppstein and Hirschberg \cite{EH2015} given in Theorem \ref{eh} is in fact proved there for that more general model. We also mention that Borzyszkowski \cite{B2014} deals with a fourth model, where the answer is yes if all balls have the same color in the query and no otherwise, like in the general model, but in that case also two balls of different colors in that query are pointed out. He determined exactly the number of queries needed to ask in that model, and this implies the lower bound in Theorem \ref{borz}.

\subsection{New results}

In this subsection we state our new results. Before that, we recall some definitions.

A hypergraph has Property B - introduced by Bernstein \cite{B1908} - if its vertices can be colored with two colors, such that there is no monochromatic edge in the hypergraph. Let us denote by $\mathsf{m}(k)$ the cardinality of the edge set of a smallest $k$-uniform hypergraph that does not have Property B. As far as we know the best known recent upper bound on $\mathsf{m}(k)$ is $O(2^k\sqrt{k/\log{k}})$ due to Radhakrishnan and Srinivasan \cite{RS2000}.

\begin{thm}\label{ayesno} For any $3 \le k \le n$, we have $$\frac{3n+5}{4} \le A(GM,k,n)\le n-k+\mathsf{m}(k).$$

\end{thm}

\begin{thm}\label{acm} For any $3 \le k \le n$, we have $$A(CM,k,n)\ge \frac{6n}{5k+6}-1. $$

\end{thm}

\section{Proof of Theorem \ref{ayesno}}

\subsection{Upper bound}

We will improve the upper bound using Property B. We start with $\mathsf{m}(k)$ queries that form a hypergraph without Property B. Then we must get a yes answer to one of them. We take a subset of size $k-1$ of that query, and another ball. This way we find out if that ball is of the same color as the $k-1$-set. Repeating this for every ball, we can identify the color classes.
\qed

\subsection{Lower bound}

\subsubsection{Introduction, Statements}


We present a strategy of Adversary. In his strategy Adversary does not only answer yes or no to a query, but possibly also tells the color of some balls. If he answers 'yes' to a query (that means all the balls have the same color in that query), then he tells the color of the balls in that query. In some other cases, he tells the colors of some of the balls. As he gives more information, a lower bound for the number of the queries asked provides a lower bound for $A(GM,k,n)$. 

The aim of telling the color of some balls is to build extra structure: to be able to control the possible colors of those balls, whose color is yet unknown at a certain step. 

\
 
During the description of the strategy we will use (possibly with some index) small letters for balls, capital letters for sets of balls and calligraphic letters for families of sets of balls. Let $Q_i$ be the query asked in the $i^{th}$ round and the set of queries during the first $i$ rounds is $\cQ(i):=\{Q_1,Q_2, \ldots, Q_i\}$.

\

\textbf{Colorings.} As we mentioned, in each round the answer of Adversary consists of a yes or no response and the color of some (possibly zero) balls. Let $S \subset [n]$ be a component of the query hypergraph $([n],\cQ(i))$ (i.e. a containment minimal subset $X$ of the vertices such that for every $Q \in \cQ(i)$, either $Q \subset X$, or $Q \cap X = \emptyset$). Let us denote by $con(S,i)$ those colorings of the balls in $S$, that are consistent with the answers for those queries $Q \in \cQ(i)$, that are contained in $S$. Since queries in different components have no effect on each other, we have the following :

\begin{prop}\label{unionconsistent} Let $S$ and $S'$ be union of components of the query hypergraph $([n],\cQ(i))$ with $S \cap S'=\emptyset$, and suppose that $c$ is a consistent coloring of $S$ and $c'$ is a consistent coloring of $S'$. Then the union of these colorings is a consistent coloring of $S \cup S'$.

\end{prop}

We say that \textit{we know the color of the ball x} 
after the $i^{th}$ round if $c(x)$ is the same for all $c \in con([n],i)$. Let us denote by $R(i)$ (resp. $B(i))$ the set of balls that are known to be red (resp. blue) after the $i^{th}$ round. Note that if Adversary never gives us the color of some balls as additional information, then we cannot know the color of any ball, as changing the color of all the balls is consistent with all the yes/no responses. 

\

\textbf{Different type of queries and balls after round $i$} 

Now we introduce some notation. We use it to describe the structure that Adversary will maintain during his strategy.

\vspace{2mm}

$\bullet$ $\cQ_r(i):=\{Q \in \cQ(i) : Q \cap R(i) \neq \emptyset, \ Q \not \subset R(i), \ Q \cap B(i) = \emptyset \}$, the set of queries that contain red but no blue balls and are not subset of $R(i)$.

\vspace{1mm}

$\bullet$ $\cQ'_r(i):=\{Q \setminus R(i) : Q \in \cQ_r(i) \}$, those parts of the previous queries that are not necessarily red. Note that by definition each of these contains at least unknown ball that will turn out to be a blue ball at the end of the algorithm (but we do not know which one it is).

\vspace{1mm}

$\bullet$ $X_r(i):= \{x \in [n] : \exists P \in \cQ'_r(i) \ \textrm{with} \ x \in P \}(=\cup\cQ'_r(i))$

\vspace{2mm}
Now we define the similar notions for color blue:
\vspace{1mm}

$\bullet$ $\cQ_b(i):=\{Q \in \cQ(i) : Q \cap B(i) \neq \emptyset, \ Q \not \subset B(i), \ Q \cap R(i) = \emptyset \}$

\vspace{1mm}

$\bullet$ $\cQ'_b(i):=\{Q \setminus B(i): Q \in \cQ_b(i) \}$

\vspace{1mm}

$\bullet$ $X_b(i):=  \{x \in [n] : \exists P \in \cQ'_b(i) \ \textrm{with} \ x \in P \}(=\cup\cQ'_b(i))$

\vspace{3mm}

$\bullet$ $\cQ_d(i):=\{Q \in \cQ(i) : Q \cap (R(i) \cup B(i))= \emptyset\}$, the set of queries where all balls are of unknown color.

\vspace{1mm}

$\bullet$ $X_d(i):= \{x \in [n] : \exists P \in \cQ_d(i) \ \textrm{with} \ x \in P \}(=\cup\cQ_d(i))$

\vspace{3mm}

$\bullet$ $\cQ_0(i):=\{Q \in \cQ(i) : Q \cap R(i) \neq \emptyset, \ Q \cap B(i) \neq \emptyset \ \textrm{or} \ Q \subset R(i) \ \textrm{or} \ Q \subset B(i) \}$, the set of queries that do not give any more information once we know $R(i)$ and $B(i)$.

\vspace{1mm}

$\bullet$ $X_0(i):= [n] \setminus (R(i) \cup B(i) \cup X_r(i)\cup X_b(i)\cup X_d(i)$, the set of remaining balls.

\vspace{2mm}

\begin{figure}
\centering
    \includegraphics[width=0.8\textwidth]{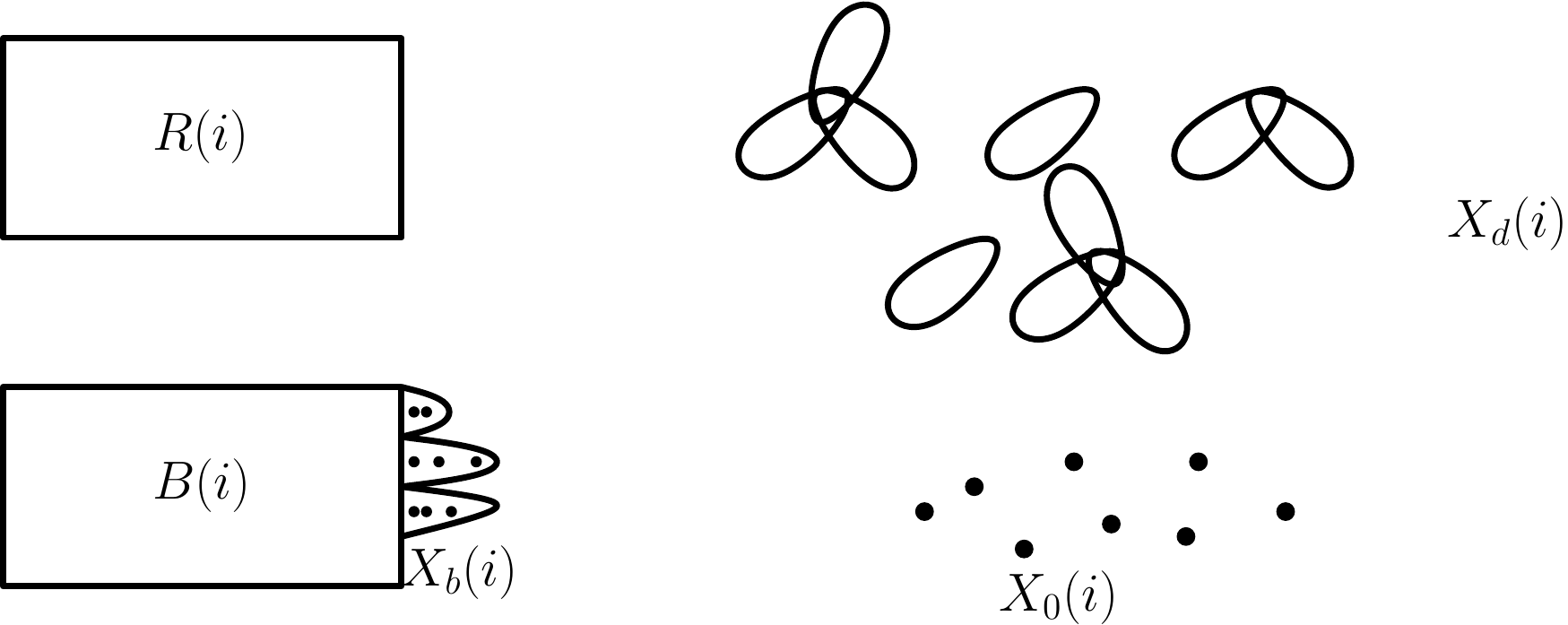}
        \caption{Partition of the balls after $i$ round using Adversary's strategy}
        \label{fig:eps1}
\end{figure}
\vspace{3mm}

\textbf{Overview of the strategy.} Adversary's main purpose is ensuring that a majority ball is found only if almost every ball is in $R(i)$ or $B(i)$. To do so, Adversary makes $R(i)$ and $B(i)$, the set of red and blue balls roughly equal size, and maintains strong structure conditions on the sets in $\cQ'_b(i)$, $\cQ'_r(i)$ and $\cQ_d(i)$, to be able to control the balls whose color is unknown. 

\ 

The strategy of Adversary consists of two parts that we call $\mathtt{STRATEGY1}$ and $\mathtt{STRATEGY2}$. Adversary starts with $\mathtt{STRATEGY1}$, and at one specific point he might switch to $\mathtt{STRATEGY2}$ and use that till the end of the process. That specific point is described in case 2 of Claim \ref{clm:endstrategy1}. If $\mathtt{STRATEGY1}$ would lead to that point, Adversary immediately aborts $\mathtt{STRATEGY1}$, picks the opposite answer and switches to $\mathtt{STRATEGY2}$.

\ 

Let us briefly overview the structural properties of the query hypergraph that Adversary will maintain during $\mathtt{STRATEGY1}$ (see Figure \ref{fig:eps1}), and then describe more precisely in Lemma \ref{lem:strategy1} below. First of all, Adversary wants $R(i)$ and $B(i)$ to have roughly equal size (Lemma \ref{lem:strategy1}: 1.). Then he maintains three properties of the queries in $\cQ_r(i)\cup \cQ_b(i)$, i.e. those queries that intersect one of $B(i)$ or $R(i)$, but are not contained in it. First, he wants that the part of them not in $R(i)$ or $B(i)$ (i.e. the part in $\cQ'_b(i)$) should be large (which means size at least $2$, Lemma \ref{lem:strategy1}: 2.). Also, those parts should be disjoint from each other (Lemma \ref{lem:strategy1}: 4.). Finally, such queries exist just for one of $R(i)$ or $B(i)$; the smaller one (Lemma \ref{lem:strategy1}: 3.). For the components of the remaining queries (where we do not know the color of any ball) Adversary maintains that they form stars, which means there is an element $x$ (called the center of the star), such that every pairwise intersection of queries in that component is $\{x\}$ (Lemma \ref{lem:strategy1}: 6.). He also maintains that the queries in these components are disjoint from those queries that intersect $B(i)$ or $R(i)$ (Lemma \ref{lem:strategy1}: 5.).
 
\ 

We state the lemma about $\mathtt{STRATEGY1}$ with the above mentioned properties and we will define $\mathtt{STRATEGY1}$ during the proof of 
Lemma \ref{lem:strategy1}.

\begin{lemma}\label{lem:strategy1}

Using $\mathtt{STRATEGY1}$ Adversary can answer in the first $i$ rounds in such a way that for all $i \ge 1$ we have: 

\vspace{2mm}

1. $\Big{|}|R(i)|-|B(i)|\Big{|} \le 1$,

\vspace{1mm}

2. for $Q \in \mathcal{Q}_r'(i) \cup \mathcal{Q}_b'(i)$ we have $|Q| \ge 2$,

\vspace{1mm}

3. either $\mathcal{Q}_r'(i)$ or $\mathcal{Q}_b'(i)$ is empty; if $|R(i)|>|B(i)|$ then $\mathcal{Q}_r'(i)=\emptyset$ and vice versa,

\vspace{1mm}

4. for different $P, Q \in \mathcal{Q}_r'(i) \cup \mathcal{Q}_b'(i)$, we have $P \cap Q = \emptyset$,

\vspace{1mm}

5. for $P \in \mathcal{Q}_r'(i) \cup \mathcal{Q}_b'(i)$ and $ Q \in \mathcal{Q}_d(i)$, we have $P \cap Q = \emptyset$,

\vspace{1mm}

6. for $Q \in \cQ_d(i) \ \exists q \in Q$ such that $Q \setminus \{q\}$ is disjoint from other members of $\mathcal{Q}_d(i)$  

(or equivalently we can say that each component is a star).

\end{lemma}

\noindent
Note that by Lemma \ref{lem:strategy1}, $X_0(i)$, $X_d(i)$ and $R(i) \cup B(i) \cup X_b(i)$ are unions of components of $([n],\cQ(i))$ with $[n]=X_0(i) \cup X_d(i) \cup (R(i) \cup B(i) \cup X_b(i))$, and they are disjoint. So if we provide consistent colorings of these components, this will provide a consistent coloring of all the balls. Also note that in $R(i) \cup B(i) \cup X_b(i)$ we only have to color the balls in $X_b(i)$. 
 
We postpone the proof of Lemma \ref{lem:strategy1} to Subsection \ref{prlem}. Before that, we list some basic observations about consistent colorings of the components of $([n],\cQ(i))$, then we examine how an algorithm can end during $\mathtt{STRATEGY1}$, state a similar lemma about $\mathtt{STRATEGY2}$, and  examine how an algorithm can end during $\mathtt{STRATEGY2}$. 

The following observations are consequences of the structure that is maintained by Lemma \ref{lem:strategy1}.

\begin{obs}\label{conscol0}

For any coloring $c$ of $X_0(i)$ we have $c \in con(X_0(i),i)$.

\end{obs}

\begin{proof}

As there are no restrictions for the color of these balls, any coloring is a consistent coloring.

\end{proof}

\begin{obs}\label{conscold} If $X_d(i) \neq \emptyset$, then:

a) for any $x \in X_d(i)$, there is a consistent coloring $c$ of $X_d(i)$, such that $c(x)$ is blue, but there are more red balls than blue (among the balls in $X_d(i)$),

b) there is a consistent coloring of $X_d(i)$ with more blue balls than red balls. 

\end{obs}

\begin{proof}

We know that the connected components of the queries in $\cQ_d(i)$ are stars (Lemma \ref{lem:strategy1}: 6.), and we also know that the answers for the queries in $\cQ_d(i)$ were no.

\smallskip 

To prove a) first let us suppose that $x$ is the center of a star. Then color the center of all stars blue and all the remaining balls red. As $k \ge 3$, we are done in this case.

If $x$ is not a center, then color the center of every component (that are stars) red, and in the remaining part of each query color $1$ ball blue (in the part containing $x$ it should be $x$), and at least $1$ ball red (we again use that $k \ge 3$). 

\smallskip

To prove b) just color the center of every star blue and every other ball red.

\end{proof}

\begin{obs}\label{conscolb} Let us suppose that there is $P \in \cQ'_b(i)$ with $|Q| \neq 2$. Then:

a) for any $x \in X_b(i)$, there is a consistent coloring $c$ of $X_b(i)$, such that $c(x)$ is blue, but there are at least as many red colored balls as blue colored (among the balls in $X_b(i)$), 

b) there is a consistent coloring of $X_b(i)$ with all the balls colored red,

c) there is a consistent coloring of $X_b(i)$ with more blue colored balls than red colored.

\end{obs}

\begin{proof}
To prove a), just color $x$ blue and every other ball in $X_b(i)$ red.

\smallskip

The proof of b) is obvious.

\smallskip

To prove c) just color one ball from every $P \in \cQ'_b(i)$ red and all the remaining balls from $X_b(i)$ blue. As we know that there is $P \in \cQ'_b(i)$ with $|P| \ge 3$, we are done.
\end{proof}

Now we prove a claim about the case we can find a majority ball during $\mathtt{STRATEGY1}$. We note that to prove Claim \ref{clm:endstrategy1}, we do not use all the properties of $\mathtt{STRATEGY1}$ given in Lemma \ref{lem:strategy1} (some of those properties were there just to maintain the structure in Lemma \ref{lem:strategy1}).

\begin{clm}\label{clm:endstrategy1} During $\mathtt{STRATEGY1}$ we can choose a majority ball or show that there is no majority color only in the following 2 cases:

\vspace{1mm}

1.  $|R(i)|+|B(i)|\ge n-1,$ or

\vspace{1mm}

2.  $|R(i)|=|B(i)|+1$, all $Q \in \cQ'_b(i)$ have $|Q|=2$ (or the same with blue and red reversed), and $|R(i)|+|B(i)|+|X_b(i)|=n$.

\end{clm}

\begin{proof} 

We prove Claim \ref{clm:endstrategy1} by contradiction. Let us suppose that neither $1.$ nor $2.$ holds, and by symmetry and $3.$ of Lemma \ref{lem:strategy1}, we can suppose that we have $|R(i)| \ge |B(i)|$ and $\mathcal{Q}'_r(i)=\emptyset$.
In the following we go through different cases and give consistent colorings of the connected components. Note that by Proposition \ref{unionconsistent} this gives us a consistent coloring of all balls.

\

\textbf{Case 1:} suppose we point to a ball $x \in X_0(i)$ as a majority ball.  Now we give a consistent coloring of the balls, in which $x$ is not a majority ball:

$\bullet_1$ color $x$ blue and every other ball in $X_0(i)$ red (by Observation \ref{conscol0} a)),

$\bullet_2$ color the points $X_d(i)$ (if it is not empty) in such a way that there are more red balls there (by Observation \ref{conscold} a)),

$\bullet_3$ color the balls in $X_b(i)$ (if it is not empty) red (by Observation \ref{conscolb} b)).

\smallskip

\noindent
As $1.$ does not hold we know that either $X_b(i)$ or $X_d(i)$ (or both) is not empty, or $|X_0(i)| \ge 2$. So either there are more red balls in $X_b(i)$ or $X_d(i)$ (or both) and only one blue in $X_0(i)$. Or there are at least as many red balls in each part as blue. So, using that $|R(i)| \ge |B(i)|$, we know that $x$ can not be a majority ball in neither cases.

\ 

\textbf{Case 2:} suppose we point to a ball $x \in X_d(i)$ as a majority ball. Now we give a consistent coloring of the balls, in which $x$ is not a majority ball:

$\bullet$  every ball in $X_0(i)$ red (by Observation \ref{conscol0}),

$\bullet$ color $x$ blue and color the balls $X_d(i)$ in such a way that there are more red balls (by Observation \ref{conscold} a)),

$\bullet$ color the balls in $X_b(i)$ red (by Observation \ref{conscolb} b)).

\smallskip

\noindent
In each part there will be at least as many red ball as blue, so a blue ball can not be majority ball. 

\ 

\textbf{Case 3:} suppose we point to a ball $x \in X_b(i)$ as a majority ball. Now we give a consistent coloring of the balls, in which $x$ is not a majority ball:

$\bullet$  every ball in $X_0(i)$ red (by Observation \ref{conscol0}),

$\bullet$ color the balls $X_d(i)$ (if it is not empty) in such a way that there are more red balls (by Observation \ref{conscold} a)),

$\bullet$ color $x$ blue and the balls in $X_b(i)$ in such a way that there are at least as many red balls in $X_b(i)$ as blue balls (by Observation \ref{conscolb} a)).

\smallskip

\noindent
In each part there will be at least as many red ball as blue, so a blue ball can not be majority ball. 

\ 

\textbf{Case 4:} suppose we point to a ball $x \in B(i)$ as a majority ball. Now we give a consistent coloring of the balls, in which $x$ is not a majority ball:

$\bullet$ color every ball in $X_0(i)$ red (by Observation \ref{conscol0}),

$\bullet$ color the points $X_d(i)$ (if it is not empty) in such a way that there are more red balls (by Observation \ref{conscold} a)),

$\bullet$ color the balls in $X_b(i)$ red (by Observation \ref{conscolb} b)).

\smallskip

\noindent
In each part there will be at least as many red ball as blue, so a blue ball can not be majority ball. 

\ 

\textbf{Case 5:} suppose we point to a ball $x \in R(i)$ as a majority ball. Now we give a consistent coloring of the balls, in which $x$ is not a majority ball:

$\bullet_1$ color every ball in $X_0(i)$ blue (by Observation \ref{conscol0}),

$\bullet_2$ color the points $X_d(i)$ (if it is not empty) in such a way that there are more blue balls (by Observation \ref{conscold} b)), 

$\bullet_3$ color the balls in $X_b(i)$ in such a way that there at least as many blue balls as red; if there is $Q \in \cQ'_b(i)$ with $|Q| \ge 3$, color the balls in $X_b(i)$ with more blue colored balls than red colored. (by Observation \ref{conscolb} c)).

\smallskip

\noindent
We know by Lemma $\ref{lem:strategy1}$ that $|R(i)| \le |B(i)|+1$, and we know by $\bullet_1$, $\bullet_2$ and $\bullet_3$ that $X_0(i), X_b(i)$ and $X_d(i)$ contains as many blue balls as red ones. 

If either $X_d(i)$ or $X_0(i)$ is not empty then
we are done by $\bullet_1$ or $\bullet_3$. If both of them is empty, then, since $2.$ does not hold then we are done by $\bullet_3$.

\ 

\textbf{Case 6:} finally suppose that we state that there is no majority ball. Now we give a consistent coloring of the balls, there exists a majority ball:

$\bullet_1$ color every ball in $X_0(i)$ red (by Observation \ref{conscol0}),

$\bullet_2$ color the points $X_d(i)$ (if it is not empty) in such a way that there are more red balls (by Observation \ref{conscold} a)),

$\bullet_3$ color the balls in $X_b(i)$ red (by Observation \ref{conscolb} b)).

\smallskip

\noindent
We know that some of $X_0(i), X_d(i)$ or $X_b(i)$ is not empty and that means that there will be more red balls at the end.

\end{proof}

We mention again that Case 2 of Claim \ref{clm:endstrategy1} will not actually occur, as Adversary switches to $\mathtt{STRATEGY2}$ in case it would happen. This is what we call the end of $\mathtt{STRATEGY1}$.

\begin{lemma} \label{lem:strategy2} If Case 2 of Claim \ref{clm:endstrategy1} would happen in the $j$th round during $\mathtt{STRATEGY1}$, then from the $j$th round Adversary could answer - following $\mathtt{STRATEGY2}$ - in a way that in the $i$th round (for $j\le i$) the following hold:

\vspace{2mm}

1. $|R(i)|+1=|B(i)|,$

\vspace{1mm}

2. for $Q \in \cQ'_b(i)$ we have $|Q|=2,$ and

\vspace{1mm}

3. $|R(i)|+|B(i)|+|X_b(i)|=n$.

\end{lemma}

\noindent
Now we prove a claim about the end of the algorithm during $\mathtt{STRATEGY2}$. The situation is kind of analogous to Claim \ref{clm:endstrategy1}:

\begin{clm}\label{cl21} During $\mathtt{STRATEGY2}$ we can choose a majority ball or show that there is no majority color only in case $|R(i)|+|B(i)| = n$.
\end{clm}

\begin{proof} We prove Claim \ref{clm:endstrategy1} by contradiction. Let us suppose that $|R(i)|+|B(i)| < n$. In each of the following cases we give a consistent coloring of $X_b(i)$, that is a consistent coloring of all the balls, that will lead to a contradiction.

\medskip

\textbf{Case 1:} suppose we point a ball $x \in X_b(i)$ as a majority ball. Now we give a consistent coloring of the balls, in which $x$ is not a majority ball:

$\bullet$ we color in every $P \in \cQ'_b(i)$ 1 ball (including $x$) red and 1 ball blue.

\medskip

\textbf{Case 2:} suppose we point to a ball $x \in R(i)$ as a majority ball. Now we give a consistent coloring of the balls, in which $x$ is not a majority ball:

$\bullet$ we color in every $P \in \cQ'_b(i)$ 1 ball red and 1 ball blue.

\medskip

\textbf{Case 3:} suppose we point to a ball $x \in B(i)$ as a majority ball. Now we give a consistent coloring of the balls, in which $x$ is not a majority ball:

$\bullet$ we color every ball in $X_b(i)$ red.

\medskip

\textbf{Case 4:} suppose we do not point to any balls. Now we give a consistent coloring of the balls, in which $x$ is not a majority ball:

$\bullet$ we color every ball in $X_b(i)$ red.

\end{proof}

\noindent
Now we know by Claim \ref{clm:endstrategy1} and Claim \ref{cl21}, that we can choose a majority ball only if we know the color of almost all the balls, so Adversary's strategy
can be that he colors as few balls as possible in each round.
The following definition captures this property.

\begin{defn}
We call the $i^{th}$ round a $(j,k)$-step, if Adversary puts $k$ queries into $\cQ_0(i)$ (that were not in $\cQ_0(i-1)$) with the coloring of $j$ balls (i.e. altogether $j$ new balls are put into $R(i) \cup B(i)$).

\end{defn}

\begin{lemma}\label{lem:jkstep}

Adversary can answer during $\mathtt{STRATEGY1}$ and $\mathtt{STRATEGY2}$ such that each round is a $(j,k)$-step with $j/k \le 4/3$, or a $(0,0)$-step.

\end{lemma}

Note that if a ball is in $R(i)$, then it is in $R(j)$ for $j>i$ (similarly, if a ball is in $B(i)$, then it is in $B(j)$ for $j>i$). Lemma \ref{lem:jkstep} with Case 1 of Claim \ref{clm:endstrategy1} and with Claim \ref{cl21} immediately implies Theorem \ref{ayesno}. Indeed, if step $s$ is the last one, altogether at least $n-1$ balls must be put into $R(i)\cup B(i)$ for some $i\le s$, thus at least $3(n-1)/4$ queries must be put into $\cQ_0(i)$ for $i\le s$ altogether. Note that every query is put into $\cQ_0(i)$ for some $i$ at most once. Thus there are at least $3(n-1)/4$ queries that do not result in $(0,0)$-steps. In addition it is going to be obvious from the description of $\mathtt{STRATEGY1}$ that the first two queries are $(0,0)$-steps.

\subsubsection{Proof of the lemmas}\label{prlem}

\begin{proof}

In this section we will simultaneously prove Lemma \ref{lem:strategy1}, Lemma \ref{lem:strategy2} and Lemma \ref{lem:jkstep} by defining $\mathtt{STRATEGY1}$ and $\mathtt{STRATEGY2}$.

We define what Adversary answers in the $i^{th}$ round ($i\ge 1$) by a case-by-case analysis of the intersection of the new query $Q_i$ and the earlier queries. In each case we provide a description of the answer which contains: the answer of Adversary (yes/no), the color of some balls, which queries (of $\cQ(i)$) go to $\cQ_0(i)$ after this round, can it be the end of $\mathtt{STRATEGY1}$ and for which $j$ and $k$ it is a $(j,k)$-step.
 
\vspace{2mm}

Recall that by symmetry we can assume that $\mathcal{Q}_r'(i-1)$ is empty (and $|R(i-1)|\ge |B(i-1)|$). Now we describe how Adversary answers to $Q_i$. 

\vspace{2mm}

About the proof of the lemmas:

$\bullet$ In each case maintaining properties 1,2,4,5 and 6 of Lemma \ref{lem:strategy1} will be obvious from the description. We illustrate this in Case 1/E, but omit the details in all the other cases.

$\bullet$ After the $i^{th}$ round, property 3 of Lemma \ref{lem:strategy1} could be violated by two reasons: the first one is that neither $\mathcal{Q}_b'(i)$ nor $\mathcal{Q}_r'(i)$ is empty. In this situation we color equal many balls (one in each set) in $\cQ'_r(i)$ blue and $\cQ'_b(i)$ red so that either $\mathcal{Q}_b'(i)$ or $\mathcal{Q}_r'(i)$ becomes empty. We color $2t$ balls for some $t$, and we can move the corresponding $2t$ queries to $\cQ_0(i)$. The other reason is that $|R(i)|$ becomes larger than $|B(i)|$, while $\mathcal{Q}_r'(i)$ is not empty. Note that by the above, we can assume that $\mathcal{Q}_b'(i)$ is empty. Then, in this situation we can pick an arbitrary member of $\mathcal{Q}_r'$ and color one of its elements blue. By repeating this, either $\mathcal{Q}_r'(i)$ becomes empty, or $|B(i)|$ becomes at least as large as $|R(i)|$. Here we move $t$ queries to $\cQ_0(i)$ by coloring $t$ balls. If any of these occurs at a round, this changes a $(j,k)$-step to a $(j+t,k+t)$-step. It is easy to see that if $j/k \le 4/3$, then $(j+t)/(k+t)\le 4/3$. Note that this cannot be the end of $\mathtt{STRATEGY1}$, as some balls move to $X_0(i)$.

\vspace{4mm}

$\mathtt{STRATEGY1}$

\vspace{3mm}

\textbf{Case 1:}  $Q_i \cap R(i-1) \neq \emptyset$. 

\vspace{2mm}

\hspace{3mm} \textbf{Case 1/A:} $Q_i \subset R(i-1)$. 

\vspace{2mm}
The answer of Adversary is yes, we do not color any ball (since the color of the balls in $Q_i$ is red), only $Q_i$ becomes a new element of $\cQ_0(i)$, so it is a $(0,1)$-step, and it can not be the end of $\mathtt{STRATEGY1}$.

\vspace{3mm}

\hspace{3mm} \textbf{Case 1/B:} $Q_i \cap B(i-1) \neq \emptyset$. 

\vspace{2mm}

The answer of Adversary is no, we do not color any ball, only $Q_i$ becomes a new element of $\cQ_0(i)$, so it is a $(0,1)$-step, and it can not be the end of $\mathtt{STRATEGY1}$.

\vspace{3mm}

\hspace{3mm} \textbf{Case 1/C:} $Q_i \cap B(i-1)=\emptyset, \ Q_i \cap X_0(i-1) \neq \emptyset.$ 

\vspace{2mm}

The answer of Adversary is no, we color one ball in $Q_i \cap X_0(i-1)$ blue, only $Q_i$ becomes a new element of $\cQ_0(i)$, so it is a $(1,1)$-step, and it can not be the end of $\mathtt{STRATEGY1}$.

\vspace{3mm}

\hspace{3mm} \textbf{Case 1/D:} $Q_i \cap B(i-1)=\emptyset, \ Q_i \cap X_0(i-1)= \emptyset, \ Q_i \cap X_b(i-1) \neq \emptyset$. 

\vspace{1mm}

In this case we know that $Q_i$ intersects a set $P' \in \cQ_b'(i-1)$ that is a subset of a query $P \in \cQ_b(i-1)$, and we know that $|P'|\ge 2$.

\vspace{2mm}

The answer of Adversary is no, we color a ball in $Q_i \cap P'$ blue and another in $P'$ red (using that $|P'| \ge 2$), thus $Q_i$ and $P$ become new elements of $\mathcal{Q}_0(i)$, so it is a (2,2)-step, and it can not be the end of $\mathtt{STRATEGY1}$. 

\vspace{3mm}

\hspace{3mm} \textbf{Case 1/E:} $Q_i \cap B(i-1)=\emptyset, \ Q_i \cap X_{0}(i-1)= \emptyset, \ Q_i \cap X_b(i-1) = \emptyset, \ Q_i \cap X_{d}(i-1) \neq \emptyset$. 

\vspace{1mm}

We know that there is $p \in Q_i \cap X_d(i-1)$. Thus there is $P \in \cQ_{d}(i-1)$ with $p \in P$.

\vspace{2mm}

If the component of $P$ contains more queries, there is a center in it. If $p$ is not the center, let $q$ be the center, otherwise let $q \in P \setminus \{p\}$ be arbitrary.
The answer of Adversary is no, we color $p$ blue and $q$ red, thus $Q_i$ and $P$ becomes a new element of $\mathcal{Q}_0(i)$. If there are other queries in the component of $P$, those other queries move to either $\cQ_r(i)$ or $\cQ_b(i)$. It is a $(2,2)$-step, and it can not be the end of $\mathtt{STRATEGY1}$.

\vspace{2mm}

Here we show why properties 1,2,4,5 and 6 of Lemma \ref{lem:strategy1} are satisfied in this case. We know that they are satisfied after round $i-1$, monitor only the changes. Adversary colors $1$ ball red and $1$ ball blue, this implies the first property. If an earlier query $Q_j$ is moved to $\cQ_r(i)$ or $\cQ_b(i)$, it was in $\cQ_d(i-1)$, and only one of its elements gets colored in this round, thus its uncolored part has size at least $2$, proving the second property. 

The only new query $Q_i$ is moved to $\mathcal{Q}_0(i)$, not $\mathcal{Q}'_b(i)\cup \cQ'_r(i)$. As sets in $\mathcal{Q}'_b(i-1)\cup \cQ'_r(i-1)\cup \cQ_d(i-1)$ are pairwise disjoint, the only way the fourth property could be violated is if two queries would be moved from $\cQ_d(i-1)$ to $\mathcal{Q}_b(i)\cup \cQ_r(i)$ that intersect in an uncolored vertex, but it is impossible, as they only intersect in the center, which becomes colored.
Similarly, if there is any set in $\cQ'_r(i)\setminus \cQ'_r(i-1)$ or $\cQ'_b(i)\setminus \cQ'_b(i-1)$, then all the queries from the component of $P$ are moved there, thus they are disjoint from other queries in $\cQ_d(i)$, showing the fifth property holds.

Finally, there is no query in $\cQ_d(i)\setminus \cQ_d(i-1)$, thus the sixth property holds.

\vspace{1cm}

\textbf{Case 2:} $Q_i \cap R(i-1) = \emptyset, \ Q_i \cap B(i-1) \neq \emptyset$. 

\vspace{2mm}

\hspace{3mm}\textbf{Case 2/A:} $Q_i \subset B(i-1)$.

\vspace{2mm}

The answer of Adversary is yes, we do not color any new ball, only $Q_i$ becomes a new member of $\cQ_0(i)$, so it is a $(0,1)$-step, and it can not be the end of $\mathtt{STRATEGY1}$.

\vspace{3mm}

\hspace{3mm}\textbf{Case 2/B:} $Q_i \cap X_b(i-1) \neq \emptyset$. 

\vspace{1mm}

We know that $Q_i$ intersects a set $P' \in \cQ_b'(i-1)$ that is a subset of a query $P \in \cQ_b(i-1)$, with $|P'| \ge 2$ and let $p \in P' \cap Q_i$.

\vspace{2mm}
The answer of Adversary is no, we color $p$ red and choose a ball $q \in P' \setminus \{p\}$ that we color blue, thus $Q_i$ and $P$ becomes a new member of $\cQ_0(i)$, so it is a $(2,2)$-step, and it can not be the end of $\mathtt{STRATEGY1}$.

\vspace{3mm}

\hspace{3mm}\textbf{Case 2/C:} $Q_i \cap X_b(i-1) = \emptyset, \ Q_i \cap X_d(i-1) \neq \emptyset $.

\vspace{1mm}

We know that $Q_i$ intersects a query $P \in \cQ_d(i-1)$ in a ball $p$, and we choose a ball $q \in P \setminus \{p\}$.

\vspace{2mm}
The answer of Adversary is no, we color $p$ red and $q$ blue, thus $Q_i$ and $P$ becomes a new member of $\cQ_0(i)$ (other queries in $\cQ_d(i-1)$ can move to either $\cQ_r(i)$ or $\cQ_b(i)$), so it is a $(2,2)$-step, and it can not be the end of $\mathtt{STRATEGY1}$.

\vspace{3mm}

\hspace{3mm}\textbf{Case 2/D:} $Q_i \cap X_b(i-1) = \emptyset, \ Q_i \cap X_d(i-1) = \emptyset, \ Q_i \cap X_0(i-1) \neq \emptyset $. 

\vspace{1mm}

We have two subcases in this case:

\vspace{2mm}

\hspace{6mm}\textbf{Case 2/D/a:} $|Q_i \cap X_0(i-1)| \ge 2$. 

\vspace{2mm}

The answer of Adversary is no, we do not color anything, nothing becomes a new member of $\cQ_0(i)$ ($Q_i$ becomes an element of $\mathcal{Q}_b(i)$), so it is a (0,0)-step, and it CAN be the end of $\mathtt{STRATEGY1}$.

\vspace{1mm}

If it would be the end of $\mathtt{STRATEGY1}$, then the answer of Adversary is yes, and we continue with $\mathtt{STRATEGY2}$.

\vspace{3mm}

\hspace{6mm}\textbf{Case 2/D/b:} $|Q_i \cap X_0(i-1)| =1$.

\vspace{2mm}

The answer of Adversary is yes (note that this is the only nontrivial case when Adversary answers yes), we color that ball in $Q_i \cap X_0(i-1)$ blue, thus $Q_i$ becomes a new element of $\cQ_0(i)$, so it is a (1,1)-step and it can not be the end of $\mathtt{STRATEGY1}$.

\vspace{1cm}

\textbf{Case 3:} $Q_i \cap R(i-1)= \emptyset,\ Q_i \cap B(i-1)= \emptyset, \ |Q_i \cap X_d(i-1)| \ge 2$. 

\vspace{3mm}

\hspace{3mm}\textbf{Case 3/A:} $Q_i$ intersects at least two queries in $\cQ_d(i-1)$. 

Let $P_1,P_2 \in \cQ_d(i-1)$ be these queries and we can choose $p_1 \in P_1 \cap Q_i$, $p_2 \in P_2 \cap Q_i$ with $p_1 \neq p_2$. Let us also choose $q_1 \in P_1 \setminus \{p_1, p_2\}$ and $q_2 \in P_2 \setminus \{p_1, p_2\}$ such that $q_1 \neq q_2$. We can do that since $k \ge 3$.

\vspace{2mm}

The answer of Adversary is no, we color color $p_1$, $q_2$ red and $p_2$,$q_1$ blue, thus $Q_i,P_1,P_2$ become new elements of $\cQ_0(i)$, so it a $(4,3)$-step, and it can not be the end of $\mathtt{STRATEGY1}$.

\vspace{3mm}

\hspace{3mm} \textbf{Case 3/B:} $Q_i$ intersects only one query in $\cQ_d(i-1)$. 

\vspace{1mm}

Let $P \in \cQ_d(i-1)$ be that query and choose $p \in P \cap Q_i$ and $q \in P \setminus \{p\}$ such that $P$ is the only query in $\cQ_d(i-1)$ that contains $q$ (this is possible, since $k \ge 3$).  

\vspace{2mm}

The answer of of Adversary is no, we color $p$ red, $q$ blue, thus $Q_i,P$ become new elements of $\cQ_0(i)$, so it a $(2,2)$-step, and it can not be the end of $\mathtt{STRATEGY1}$.

\vspace{1cm}

\textbf{Case 4:} $Q_i \cap R(i-1)= \emptyset, \ Q_i \cap B(i-1)= \emptyset, \ |Q_i \cap X_d(i-1)| \le 1, \ Q_i \cap X_b(i-1) \neq \emptyset$. 

\vspace{1mm}

Let $P' \in \cQ'_b(i-1)$ (a hanged out part of $P \in \cQ_b(i-1)$) for which a ball $p \in Q_i \cap P'$.

\vspace{3mm}

\hspace{3mm}\textbf{Case 4/A:} $Q_i \cap X_0(i-1) \neq \emptyset$. 

\vspace{2mm}

The answer of Adversary is no, we color $p$ red, a ball from $Q_i \cap X_0(i-1)$ blue, thus  $Q_i$ and $P$ become new elements of $\cQ_0(i)$, it is a (2,2)-step and it can not be the end of $\mathtt{STRATEGY1}$. 

\vspace{3mm}

\hspace{3mm}\textbf{Case 4/B:} $Q_i \cap X_0(i-1) = \emptyset$. 

\vspace{1mm}

By the above we know that in this case we have $|Q_i \cap X_b(i)|\ge 2$, as the size of the queries is at least 3.

\vspace{3mm}

\hspace{6mm}\textbf{Case 4/B/a:} there is a $P\in \cQ_b(i-1)$ such that $|Q_i \cap P| \ge 2$, and let $p$ and $q$ be these balls.

\vspace{2mm}

The answer of Adversary is no, we color $p$ red and $q$ blue, thus $Q_i$ and $P$ become new elements of $\cQ_0(i)$, so it is a (2,2)-step and it can not be the end of $\mathtt{STRATEGY1}$.

\vspace{3mm}

\hspace{6mm}\textbf{Case 4/B/b:} There are two $P_1,P_2 \in \cQ_b(i-1)$ such that there are balls $p_1,p_2$ with $p_1 \in P'_1 \cap Q_i$ and $p_2 \in P'_2 \cap Q_i$. 

\vspace{1mm}

In this case choose $q_1 \in P'_1 \setminus \{p_1\}$ and $q_2 \in P'_2\setminus Q_i$ (this can be done, since $|P'_1|, |P'_2| \ge 2$).

\vspace{2mm}

The answer of Adversary is no, we color $q_1$ and $p_2$ red, $p_1$ and $q_2$ blue, thus $Q_i,P_1,P_2$ become new elements of $\cQ_0(i)$, it is a (4,3)-step and it can not be the end of $\mathtt{STRATEGY1}$.

\vspace{1cm}

\textbf{Case 5:} $Q_i \cap R(i-1)= \emptyset, \ Q_i \cap B(i-1)= \emptyset, \ Q_i \cap X_b(i-1) = \emptyset, |Q_i \cap X_d(i-1)| \le 1$.

\vspace{2mm}

\hspace{3mm}\textbf{Case 5/A:} $Q_i$ intersects at least two queries in $\cQ_d(i-1)$.

\vspace{2mm}

The answer of Adversary is no, we do not color any balls, no query becomes a new element of $\cQ_0(i)$, just $Q_i$ is put into $\cQ_d(i)$, thus it is a (0,0)-step and it can not be the end of $\mathtt{STRATEGY1}$.

\vspace{3mm}

\hspace{3mm}\textbf{Case 5/B:} $Q_i$ intersects one query in $\cQ_d(i-1)$, that does not intersect any other query in $\cQ_d(i-1)$.

\vspace{2mm}

Like in the previous subcase, the answer of Adversary is no, we do not color any balls, no query becomes a new element of $\cQ_0(i)$, just $Q_i$ is put into $\cQ_d(i)$, thus it is a (0,0)-step and it can not be the end of $\mathtt{STRATEGY1}$.

\vspace{3mm}

\hspace{3mm}\textbf{Case 5/C:} $Q_i$ intersects only one query $P_1 \in \cQ_d(i-1)$, that intersects another query $P_2\in \cQ_d(i-1)$.

\vspace{1mm} In this case let $\{p_1 \} = Q_i \cap P_1$ and$\{p_2 \} = P_1 \cap P_2$. We can choose $p_3 \in Q_i \setminus (P_1 \cup P_2)$ and $p_4 \in P_2 \setminus (P_1 \cup Q_i)$.

\vspace{2mm}

The answer of Adversary is no, we color $p_1,p_4$ red and $p_2,p_3$ blue, so $Q_i, P_1,P_2$ become new elements of $\cQ_0(i)$ (other queries containing $p_2$ move from $\cQ_d(i-1)$ to $\cQ_b(i)$), thus it is a (4,3)-step and it can not be the end of $\mathtt{STRATEGY1}$.

\vspace{1cm}

\textbf{Case 6:} $Q_i \subset X_0(i-1)$.

\vspace{2mm}

The answer of Adversary is no, we do not color any ball, thus no query becomes a new member of $\cQ_0(i)$, so it is a (0,0)-step and it can not be the end of $\mathtt{STRATEGY1}$.

\vspace{1cm}

\noindent
We continue with the description of $\mathtt{STRATEGY2}$:

\vspace{1mm}

Note that the only situation, when $\mathtt{STRATEGY1}$ can end is Case 2/D/a. In that situation Adversary answers no and we could choose a majority ball, so the structure of the query hypergraph would satisfy the following:

$\bullet$ $X_b(i) \cup R(i) \cup B(i)= [n]$, 

$\bullet$ $|R(i)|=|B(i)|+1$,

$\bullet$ for all $Q \in \cQ'_b(i)$ we have $|Q|=2$.

As we mentioned, Adversary answers yes instead, so the query hypergraph (after the 'yes' answer) satisfies the following:

$\bullet$ $X_b(i) \cup R(i) \cup B(i)= [n]$, 

$\bullet$ $|R(i)|+1=|B(i)|$,

$\bullet$ for all $Q \in \cQ'_b(i)$ we have $|Q|=2$.

\vspace{3mm}

\noindent
In the following we describe $\mathtt{STRATEGY2}$ again by case-by-case analysis similar to $\mathtt{STRATEGY1}$. The properties 1,2,3 of Lemma \ref{lem:strategy2} will be easy in each case.

\vspace{2mm}

$\mathtt{STRATEGY2}$

\vspace{2mm}

\textbf{Case A:} $ Q_i \cap X_b(i-1) \neq \emptyset$.

\vspace{1mm}

Let $p \in P' \cap Q_i$ with some $P' \in \cQ'_b(i-1)$ (a hanged out part of some $P\in\cQ_b(i-1)$) and choose $q \in P' \setminus \{p\}$.

\vspace{2mm}

\hspace{3mm} \textbf{Case A/a:} $Q_i \cap R(i-1) \neq \emptyset$.

\vspace{2mm}

The answer of Adversary is no, we color $p$ blue and $q$ red, so $P$ and $Q_i$ become new elements of $\cQ_0(i)$, thus it is (2,2)-step.

\vspace{2mm}

\hspace{3mm} \textbf{Case A/b:} $Q_i \cap B(i-1)  \neq\emptyset$.

\vspace{2mm}

The answer of Adversary is no, we color $p$ red and $q$ blue, so $P$ and $Q_i$ become new elements of $\cQ_0(i-1)$, thus it is (2,2)-step.

\vspace{2mm}

\hspace{3mm} \textbf{Case A/c:} $Q_i \cap R(i-1)=Q_i \cap B(i-1) = \emptyset$.

\vspace{2mm}

In this case $Q_i$ intersects another set $P'_0$ from $\cQ'_b(i-1)$ (a hanged out part of $P_0 \in \cQ_b(i-1)$). Let $p_0 \in P'_0 \cap Q_i$ and choose $q_0 \in P'_0 \setminus \{p_0\}$.

\vspace{2mm}

The answer of Adversary is no, we color $p$ and $q_0$ red, $q$ and $p_0$ blue, so $Q_i, P, P_0$ become new elements of $\cQ_0(i)$, thus it is a (4,3)-step.

\vspace{3mm}

\textbf{Case B:} $Q_i \cap X_b(i) = \emptyset$.

\vspace{2mm}

\hspace{3mm} \textbf{Case B/a:} $Q_i \subset R(i)$.

\vspace{2mm}

The answer of Adversary yes, we do not color any new ball, $Q_i$ becomes a new element of $\cQ_0(i)$, thus it is a (0,1)-step.

\vspace{3mm}

\hspace{3mm} \textbf{Case B/b:} $Q_i \subset B(i)$.

\vspace{2mm}

The answer of Adversary yes, we do not color any new ball, $Q_i$ becomes a new element of $\cQ_0(i)$, thus it is a (0,1)-step.

\vspace{3mm}

\hspace{3mm} \textbf{Case B/c:} $Q_i \cap R(i) \neq \emptyset$, $Q_i \cap B(i) \neq \emptyset$.

\vspace{2mm}

The answer of Adversary no, we do not color any new ball, $Q_i$ becomes a new element of $\cQ_0(i)$, thus it is a (0,1)-step.

\vspace{3mm}

\end{proof}

\section{Proof of Theorem \ref{acm}}

Now we describe Adversary's strategy, that will contain the answer and the color of some balls as additional information.
Let us denote by $R(i)$ (resp. $B(i))$ the set of balls that are known to be red (resp. blue) after the $i^{th}$ round, and let $g(i):=||R(i)|-|B(i)||$ be their difference. Let $x:=\lfloor (k-1)/3 \rfloor$. At any time during the algorithm we say that a query is \textit{open} if it contains more than $k-x$ balls of degree one (in the query hypergraph till that point), and \textit{closed} otherwise. 

In each round Adversary tells the color of the balls of degree at least two, and he tells the color of all the balls in closed queries (i.e. in queries that became closed in that round). Note that a query $Q$ can also become closed in a later round, when a query $Q'$ is asked; at this point Adversary gives the color of all the balls in $Q$, even the ones not in $Q'$. For an open query $Q$ we know the color of some of its balls and also the answer to $Q$, i.e. we know that there are two possible numbers of blue balls among the uncolored balls in $Q$ (which are the balls of degree one in $Q$).

When the $i^{th}$ query $Q$ is asked, it might contain some balls already colored (those in $R(i-1)\cup B(i-1)$). It might also contain some balls that previously had degree one. Adversary will color these balls, so first we describe, how:

\vspace{2mm}

$\bullet$ If there are more than $k-x$ balls that appear in a query the first time, i.e. $Q$ is open, the answer is always $x$. The Adversary does not color the balls that will have degree one after this step. Note that balls that are not colored before the $i^{th}$ round and have degree two after this step come from previously open queries. For each such query $Q_0$, Adversary colors balls in $Q\cap Q_0$ in such a way that the coloring is consistent to the answer to $Q_0$ (which was $x$). If $Q_0$ remains open, any coloring satisfies that, Adversary chooses one that minimizes $g(i)$. If $Q_0$ becomes closed, then Adversary has to color all the balls in $Q_0$, thus choose one of two possibilities: either there are $x$ blue or $x$ red balls in $Q_0$. Both are possible, as before the $i^{th}$ round there were less than $x$ blue and less than $x$ red balls in $Q_0$ (as there were less than $x$ colored balls altogether). It does not matter which balls in $Q_0$ become red and which ones become blue in this round. Thus indeed there are two possibilities, Adversary chooses one that minimizes $g(i)$.

\vspace{2mm}

$\bullet$ If the query $Q$ is closed then Adversary will color all the balls in it. Similarly to the previous case first he colors the balls that appeared in other queries and the balls in queries that become closed. Moreover, he colors them the same way. Then he has to color the remaining balls in $Q$; he chooses a coloring that minimizes $g(i)$, and finally he answers the number that this coloring gives as the answer to $Q$.

\vspace{2mm}


We claim that at any point $g(i)\le k-2x$. Indeed, there is only one case when colors are picked and the goal is not only making $g(i)$ as small as possible. It is when an earlier query $Q_0$ becomes closed, and there are two possible choices for Adversary: after the coloring there are either $x$ red and $k-x$ blue balls or $k-x$ red and $x$ blue balls in $Q_0$. No matter how many red and blue balls are already in $Q_0$, the difference between the number of new red balls is $k-2x$. Another important observation is that one choice colors blue more than half of the newly colored balls, while the other choice colors red more than half of the newly colored balls, but their difference is at most $k-2x$. Thus if earlier there were more blue balls, or the same as red balls, Adversary colors more red, thus either decreases the difference, or pushes it into the other direction, but by no more than $k-2x$. In other cases the colors are picked without any restriction, with the goal of making $g(i)$ as small as possible, thus it cannot increase (except from 0 to 1).

Now let us assume that the algorithm has finished and let $A$ be the set of balls not appearing in any queries, let $B$ be the set of balls that have degree one in the query hypergraph (i.e. appeared in an open query), and let $C$ be the set of the remaining balls. Note that we know the colors of the balls in $C$.

Consider an open query $Q$, and let $Q_0: =Q\cap B$ (i.e. those balls in $Q$ that remained degree one). The answer for $Q$ was $x$. The way we chose $x$ and the definition of an open query shows that both colors appear in $Q_0$, and we do not know which balls are red, thus we cannot claim that a ball in $Q_0$ is a majority ball. It is also easy to see that no matter what colors the balls not in $Q_0$ have, it is consistent with the answers that $Q$ contains $x$ red balls but also that $Q$ contains $x$ blue (and $k-x$ red) balls. It means that inside $Q_0$, independently of the color of other balls, there are two possible colorings, where the difference between the color classes is at least $k-3x$, and in one of the colorings there are more blue balls, in the other coloring there are more red balls.

So we cannot claim that a ball in $B$ (or in $A$) is a majority ball. Now we show that there are at most $(k-2x-|A|)/(k-3x)$ open queries. If we could claim that there is no majority color, then even one open query would lead to contradiction, as we could change the number of blue balls among the balls of degree one in it. Assume we claim that a ball $i \in C$ is a majority ball. We know the color of $i$, say red. We also know that there are at most $k-2x$ more red than blue balls in $C$. If there are more than $(k-2x-|A|)/(k-3x)$ open queries, then blue can be the majority (if there are $x$ red balls in each of those queries and $A$ is blue), a contradiction. 

We know that $A$ and $B$ contains together at most $$k(k-2x-|A|)/(k-3x)+|A| \le k(k-2x)/(k-3x)$$ balls. The remaining $n'=n-k(k-2x)/(k-3x)$ balls must be covered by queries that contain at most $k-x$ balls of degree one. A simple computation shows that at least $2n'/(2k-x)$ queries are needed for that, which easily implies the lower bound.
\qed

\vspace{3mm}

\section{Concluding remarks}

If we replace property 6 of Lemma \ref{lem:strategy1} by the stronger condition that queries in $\cQ_d(i)$ are disjoint from each other, a similar but simpler case analysis shows that $A(GM,k,n) \ge 2n/3$. It seems plausible that if we instead replaced it with a weaker condition, allowing a richer structure, a similar but more complicated case analysis would give a stronger lower bound. We conjecture $A(GM,k,n)=(1-o(n))n$.

Another interesting question is what happens if we allow queries of different sizes. Can it help the questioner if he can pick $k$ at any point? It is especially interesting in the general model, where the length of the algorithm does not seem to depend significantly on $k$. Our lower bound still holds in this more general version.

Instead of finding a majority ball, a reasonable goal would be to find the two color classes. In fact, our algorithm for the General Model identifies the color classes. On the other hand it is easy to see that at least $n-1$ queries are needed to solve the harder problem.
For $k=2$, in the pairing model, $n-b(n)$ queries are needed to find a majority ball, as mentioned in the introduction, while it is easy to see that $n-1$ queries are needed to find the color classes. It seems that the difference is typically so small that is becomes relevant only when the upper and lower bounds are close enough. However, the adaptive algorithm for the Counting Model by Eppstein and Hirschberg \cite{EH2015} does not find the color classes. It still finds the sizes of the color classes though.

Another variant of these problems is when all the queries are fixed at the beginning, it is called the \textit{non-adaptive} version. In a forthcoming paper \cite{gv16} we investigate the non-adaptive versions of these and other models of the majority problem.

\section*{Acknowledgement}

We would like to thank the hospitality of Moscow Institute of Physics and Technology, where this work started and all participants of the Combinatorial Search Seminar at the Alfr\'ed R\'enyi Institute of Mathematics for fruitful discussions.

We also thank the reviewer for his/her comments that significantly improved our manuscript.

\end{document}